\documentclass[a4paper,12pt]{amsart}
\usepackage{amsmath}
\usepackage{amssymb}
\usepackage{ifthen}
\usepackage{graphicx}
\usepackage{float}
\usepackage{caption}
\usepackage{subcaption}
\usepackage{cite}
\usepackage{amsfonts}
\usepackage{amscd}
\usepackage{amsxtra}
\usepackage{color}

\newtheorem{theorem}{Theorem}
\newtheorem{thm}{Theorem}
\newtheorem{rem}{Remark}
\newtheorem{corollary}{Corollary}
\newtheorem{lemma}{Lemma}
\newtheorem{lem}{Lemma}

\newtheorem{conj}{Conjecture}
\newtheorem{prob}{Problem}
\theoremstyle{definition}

\newtheorem{example}{Example}
\newtheorem{case}{Case}
\newtheorem{subcase}{Subcase}
\newenvironment{pf}[1][]{ \vskip 1mm
 \noindent
 \ifthenelse{\equal{#1}{}}  {{\slshape Proof. }}  {{\slshape #1.} } }{\qed\bigskip}

\newcounter{alphabet}


\def\be{\begin{equation}}
\def\ee{\end{equation}}

\newcommand{\bee}{\begin{enumerate}}
\newcommand{\eee}{\end{enumerate}}

\newcommand{\blem}{\begin{lem}}
\newcommand{\elem}{\end{lem}}
\newcommand{\bthm}{\begin{thm}}
\newcommand{\ethm}{\end{thm}}
\newcommand{\bcor}{\begin{cor}}
\newcommand{\ecor}{\end{cor}}
\newcommand{\beg}{\begin{example}}
\newcommand{\eeg}{\end{example}}
\newcommand{\begs}{\begin{examples}}
\newcommand{\eegs}{\end{examples}}
\newcommand{\bdefe}{\begin{defn}}
\newcommand{\edefe}{\end{defn}}
\newcommand{\bprob}{\begin{prob}}
\newcommand{\eprob}{\end{prob}}
\newcommand{\bques}{\begin{ques}}
\newcommand{\eques}{\end{ques}}
\newcommand{\bei}{\begin{itemize}}
\newcommand{\eei}{\end{itemize}}
\newcommand{\bcon}{\begin{conj}}
\newcommand{\econ}{\end{conj}}
\newcommand{\bcons}{\begin{conjs}}
\newcommand{\econs}{\end{conjs}}
\newcommand{\bprop}{\begin{propo}}
\newcommand{\eprop}{\end{propo}}
\newcommand{\br}{\begin{rem}}
\newcommand{\er}{\end{rem}}
\newcommand{\brs}{\begin{rems}}
\newcommand{\ers}{\end{rems}}
\newcommand{\bo}{\begin{obser}}
\newcommand{\eo}{\end{obser}}
\newcommand{\bos}{\begin{obsers}}
\newcommand{\eos}{\end{obsers}}
\newcommand{\bca}{\begin{case}}
\newcommand{\eca}{\end{case}}
\newcommand{\bsca}{\begin{subcase}}
\newcommand{\esca}{\end{subcase}}
\newcommand{\bpf}{\begin{pf}}
\newcommand{\epf}{\end{pf}}
\newcommand{\ba}{\begin{array}}
\newcommand{\ea}{\end{array}}
\newcommand{\beq}{\begin{eqnarray}}
\newcommand{\beqq}{\begin{eqnarray*}}
\newcommand{\eeq}{\end{eqnarray}}
\newcommand{\eeqq}{\end{eqnarray*}}

\newcounter{minutes}
\divide\time by 60
\newcounter{hours}
\multiply\time by 60 \addtocounter{minutes}{-\time}
\begin{document}
\title[]{The Bohr operator on analytic functions\\ and sections}
%\thanks{File:~\jobname .tex, printed: \number\day-\number\month-\number\year%
%, \thehours.\ifnum\theminutes<10{0}\fi\theminutes }
\author[Y. Abu Muhanna]{Yusuf Abu Muhanna}
\address{Y. Abu Muhanna, Department of Mathematics, American University of
Sharjah, UAE-26666.}
\email{ymuhanna@aus.edu}
\author[R. M. Ali]{Rosihan M. Ali}
\address{R. M. Ali, School of Mathematical Sciences, Universiti Sains
Malaysia, 11800 USM Penang, Malaysia.}
\email{rosihan@usm.my}
\author[S. K. Lee]{See Keong Lee}
\address{S. K. Lee, School of Mathematical Sciences, Universiti Sains
Malaysia, 11800 USM Penang, Malaysia.}
\email{sklee@usm.my}
%\author[S. Ponnusamy]{Saminathan Ponnusamy }
%\address{S. Ponnusamy, Department of Mathematics, Indian Institute of
%Technology Madras, Chennai-600 036, India. }
%\email{samy@iitm.ac.in}

\begin{abstract}
%We show that the Bohr and the Rogosinski operators are algebraic norms on
%the class of analytic functions and use that to obtain new results on the
%classes of subordinations.

For $f(z) = \sum_{n=0}^{\infty} a_n z^n$ and a fixed $z$ in the unit disk, $|z| = r,$ the Bohr operator $\mathcal{M}_r$ is given by
\[\mathcal{M}_r (f) = \sum_{n=0}^{\infty} |a_n| |z^n| = \sum_{n=0}^{\infty} |a_n| r^n.\]
This papers develops normed theoretic approaches on $\mathcal{M}_r$. Using earlier results of Bohr and Rogosinski, the following results are readily established: if $f(z)=\sum_{n=0}^{\infty} a_{n}z^{n}$ is subordinate (or quasi-subordinate) to $h(z)=\sum_{n=0}^{\infty} b_{n}z^{n}$ in the unit disk, then
\[\mathcal{M}_{r}(f) \leq \mathcal{M}_{r}(h), \quad  0 \leq r \leq 1/3,\]
that is,
\[\sum_{n=0}^{\infty} \left\vert a_{n}\right\vert |z|^{n} \leq \sum_{n=0}^{\infty} \left\vert b_{n}\right\vert |z|^{n}, \quad 0 \leq |z| \leq 1/3. \]
Further, each $k$-th section $s_k(f) = a_0 + a_1 z + \cdots + a_kz^k$ satisfies
\[\left| s_k(f)\right| \leq \mathcal{M}_r \left( s_k(h)\right), \quad 0 \leq r \leq 1/2,\]
%\[\left\vert\sum_{n=0}^{k} a_{n} z^{n}\right\vert \leq \sum_{n=0}^{k} \left\vert b_{n}\right\vert |z|^{n}, \quad 0 \leq |z| \leq 1/2, \]
and
\[\mathcal{M}_{r}\left( s_{k}(f) \right) \leq \mathcal{M}_{r}(s_{k}(h)), \quad 0 \leq r \leq 1/3.\]
%\[\sum_{n=0}^{k} \left\vert a_{n}\right\vert |z|^{n} \leq \sum_{n=0}^{k} \left\vert b_{n}\right\vert |z|^{n}, \quad 0 \leq |z| \leq 1/3. \]
A von Neumann-type inequality is also obtained for the class consisting of Schwarz functions in the unit disk.

\end{abstract}

\maketitle

%{A lemma of rogosinski and Mertens function}

%=========================================================================

%=========================================================================

%{Primary: 31A30, 31B30, 35B5; Secondary: 30C35, 30C45, 30C80 }

%\date{\today  %June. 30, 09
%;  File: 2013(2).tex}

\pagestyle{myheadings}
\markboth{Y. Abu Muhanna, R. M. Ali, S. K. Lee}{Bohr operator on analytic functions}

\section{\textbf{Introduction}}

Let $\mathcal{A}$ be the class consisting of analytic functions $f(z) = \sum_{n=0}^{\infty} a_n z^n$ in the unit disk $\mathbb{D} = \{ z \in \mathbb{C} : |z| < 1\}.$ A classical result of Bohr \cite{BOH} states the following:

\begin{theorem}[Bohr's theorem]\label{thm:Bohr}
If $f(z) = \sum_{n=0}^{\infty} a_n z^n \in \mathcal{A}$ satisfies $|f(z)| \leq 1$ for all $z \in \mathbb{D}$, then
\begin{equation}\label{Eq1:Bohr_inequality}
\sum_{n=0}^{\infty} |a_n| |z^n| \leq 1
\end{equation}
for $|z| \leq 1/3,$ and the constant $1/3$ cannot be improved.
\end{theorem}

Bohr's initial proof had the majorant bounded by $1/6,$ which later was improved independently by M. Riesz, I. Schur and F. Wiener. The number $1/3$ is called the Bohr radius for analytic bounded functions in the unit disk $\mathbb{D}$.

It is amazing that this classical result still commands the interest of many. There is a vast literature related to the Bohr radius found in different areas of mathematics. These include the works of \cite{YM2,YM3,YM4,YM5,AliAbdulhadiNg,BDK}, and more recently, in the works of \cite{AgraMoha,AliNg,EPR,KayPon1,LiuPon, Ravi}. There are also studies on the multi-dimensional Bohr radius, see for example, \cite{Dmn, Boas, Aiz1, Aiz2, Def, Bay}, and operator-theoretic Bohr radius, for example in \cite{Dixon, PS, PS2}. A recent survey paper on Bohr's theorem can be found in \cite{YM1}.

Another notion closely related to the Bohr radius is the Rogosinski radius contained in the following result of Rogosinski \cite{Rogosinski}.

\begin{theorem}\label{thm:Rogosinskiradius}
If $f(z) = \sum_{n=0}^{\infty} a_n z^n \in \mathcal{A}$ and $|f(z)| \leq 1$ for all $z \in \mathbb{D}$, then for every $k \in\mathbb{N}_0 = \{0, 1, 2, \dots \},$ each section satisfies
\begin{equation}\label{Eq3:Rogosinski_inequality}
\left| s_{k}(f) \right| := \left| s_{k}(z;f) \right|= \left|\sum_{n=0}^{k} a_n z^n \right| \leq 1
\end{equation}
for $|z| \leq 1/2$. The constant $1/2$ cannot be improved.
\end{theorem}

The number $1/2$ in the preceding theorem is known as the Rogosinski radius. In  \cite{Aizenberg, AES}, Aizenberg \textit{et al.} extended the Rogosinski phenomenon in several complex variables. Aizenberg \cite{Aizenberg} also studied the Rogosinski radius on Hardy spaces and Reinhardt domains. For more recent advances, see for example \cite{AKP, KayPon2, LiuShangXu}.

In the following section, the Bohr operator $\mathcal{M}_r$ is introduced on Banach algebras. This normed theoretic approach is expeditiously used to establish existing and new results. Of particular interest are the recent works of Bhowmik and Das \cite[Corollary 3.2]{BhowDas} on subordination, and also of  Alkhaleefah \emph{et al.} \cite[Theorem 2.1]{AKP} on quasi-subordination.

Recall that for two analytic functions $f$ and $h$ in $\mathbb{D},$ $f$ is \emph{subordinate} to $h$, written $f \prec h$, if there is a Schwarz function $\varphi$ such that $f(z) = h(\varphi (z))$ for $z \in \mathbb{D}$ (see \cite{Duren}). Here a \emph{Schwarz function} is an analytic function $\varphi$ satisfying $\varphi(0) = 0$ and $|\varphi(z)| \leq 1$ in $\mathbb{D}$. If there also exists an analytic function $\psi$ with $|\psi(z)| \leq 1$ in $\mathbb{D}$ and $f(z)=\psi (z)h(\varphi(z))$, then $f$ is said to be \emph{quasi-subordinate} to $h$ \cite{Duren, Pommerenke}.

We shall establish the following results in this paper: if $f(z)=\sum_{n=0}^{\infty} a_{n}z^{n}$ is subordinate (or quasi-subordinate) to $h(z)=\sum_{n=0}^{\infty} b_{n}z^{n}$ in $\mathbb{D}$, then
\[\sum_{n=0}^{\infty} \left\vert a_{n}\right\vert |z|^{n} \leq \sum_{n=0}^{\infty} \left\vert b_{n}\right\vert |z|^{n}, \quad |z| \leq 1/3. \]
Further, each $k$-th section $s_k(f) = a_0 + a_1 z + \cdots + a_kz^k$ satisfies
\[\left\vert\sum_{n=0}^{k} a_{n} z^{n}\right\vert \leq \sum_{n=0}^{k} \left\vert b_{n}\right\vert |z|^{n}, \quad |z| \leq 1/2, \]
and
\[\sum_{n=0}^{k} \left\vert a_{n}\right\vert |z|^{n} \leq \sum_{n=0}^{k} \left\vert b_{n}\right\vert |z|^{n}, \quad |z| \leq 1/3. \]
A von Neumann-type inequality is also obtained for the class consisting of Schwarz functions in $\mathbb{D}$.

\section{\textbf{Bohr operator on Banach algebras}}

For a fixed $z$ in $\mathbb{D}$, let
\[\mathcal{F}_{z} = \left\{ f(z) = \sum_{n=0}^{\infty} a_n z^n: f \in \mathcal{A} \right\}.\]
The \emph{Bohr operator }$\mathcal{M}_r$ on $\mathcal{F}_{z}$, $|z|=r,$ is given by
\[\mathcal{M}_r (f) = \sum_{n=0}^{\infty} |a_n| |z^n| = \sum_{n=0}^{\infty} |a_n| r^n.\]
The following result is readily established.

\begin{theorem}\label{Prop:Operator}
For each fixed $z$ in $\mathbb{D}$, $|z|=r,$ the Bohr operator $\mathcal{M}_r$ satisfies
\begin{enumerate}
\item[(i)] $\mathcal{M}_r(f) \geq 0,$ and $\mathcal{M}_r(f) = 0$ if and only if $f \equiv 0,$

\item[(ii)] $\mathcal{M}_r(f+g) \leq \mathcal{M}_r(f) + \mathcal{M}_r(g),$

\item[(iii)] $\mathcal{M}_r(\alpha f) = |\alpha| \mathcal{M}_r(f), \quad \alpha \in\mathbb{C},$

\item[(iv)] $\mathcal{M}_r(f\cdot g)\leq \mathcal{M}_r(f)\cdot \mathcal{M}_r(g),$
\item[(v)] $\mathcal{M}_r(1) = 1.$
\end{enumerate}
\end{theorem}
Thus the space $\mathcal{F}_{z}$ with norm $\mathcal{M}_r$ constitute a Banach algebra.

%A class of analytic functions defined on $\mathbb{U}$ is said to satisfy a Bohr phenomenon if an inequality of type \eqref{Eq1:Bohr_inequality} is satisfied uniformly in $\mathbb{U}_{\rho_0} := \{ z : |z| < \rho_0 \}$ for some $0 < \rho_0 \leq 1$ and for all functions in that class. The largest such $\rho_0$ is called the Bohr radius of that class.

Paulsen and Singh \cite{PS} have extended Bohr's theorem to Banach algebras. It is known \cite{BDK}, however, that not all Banach spaces satisfy a Bohr phenomenon. For a fixed $z \in \mathbb{D},$ let $\mathcal{B}_{z} \subset \mathcal{F}_{z}$ be the set
\[\mathcal{B}_{z} = \{\phi\in \mathcal{F}_{z} : |\phi (z)|<1 \}.\]
Then $\mathcal{B}_z$ is also a Banach algebra which by Bohr's theorem, is uniformly bounded whenever $|z| \leq 1/3.$
%Thus $\mathcal{B}_z$ satisfies a Bohr phenomenon with a Bohr radius $1/3$.

A Banach algebra $B$ satisfies the von Neumann inequality if for all $f \in B$ with $\| f \| \leq 1$ and any polynomial $p$, then
\begin{equation}\label{Eq2:vonNeumann}
\| p(f(z)) \| \leq \| p \|_{\infty}.
\end{equation}
%Here $\| \cdot \|$ and$\| \cdot \|_{\infty}$, respectively denote the norm and the usual sup-norm of the Banach algebra.
In \cite{Neumann} (see also \cite{Dixon}), von Neumann showed that the above inequality holds true for the algebra $L(H)$ of bounded operators on a Hilbert space $H$. It was also conjectured that
\begin{enumerate}
\item[(a)] every Banach algebra is isomorphic to $L(H)$,

\item[(b)] every Banach algebra satisfies the von Neumann inequality \eqref{Eq2:vonNeumann}.
\end{enumerate}

By considering the Banach algebra $l_{\beta }^{1}$ of sequences given by
\begin{equation*}
l_{\beta }^{1} = \left\{ x = (x_{1}, x_{2}, x_{3},\ldots) : \frac{1}{\beta} \sum_{j=1}^{\infty} |x_{j}| < \infty \right\},
\end{equation*}
Dixon in \cite{Dixon} disproved the conjecture (a) for any $0 < \beta \leq 1.$ By using Bohr's theorem, Dixon also showed that $l_{\beta }^{1}$ satisfies the non-unital von Neumann inequality for $0 < \beta \leq 1/3,$ but fails whenever $\beta >1/3$.

\section{\textbf{The Bohr operator on classes of subordination}}

In this section, we present our main results. Here is a fundamental result for $\mathcal{M}_r$.

\begin{lemma}\label{lem:BohrSchwarz}
If $\varphi$ is a Schwarz function, then $\mathcal{M}_{r}(\varphi(z)) \leq |z|$, for $|z|\leq 1/3$.
\end{lemma}

\begin{proof}
 Since $\left\vert \varphi (z)/z \right\vert \leq 1$ in $\mathbb{D}$, Bohr's theorem readily shows $\mathcal{M}_{r}(\varphi (z)/z) \leq 1$ for $0 \leq r\leq 1/3$.
\end{proof}

The following result on subordination was first proved by Bhowmik and Das in \cite[Corollary 3.2]{BhowDas}. We give an alternate proof using the Bohr operator $\mathcal{M}_r$.
%One can see that the result resembles the von Neumann inequality in \eqref{Eq2:vonNeumann}.

\begin{theorem}\label{thm:sub}
If $f \prec h$ in $\mathbb{D}$, then
\[\mathcal{M}_{r}(f) \leq \mathcal{M}_{r}(h), \quad  0 \leq r \leq 1/3.\]
\end{theorem}

\begin{proof}
Let $h(z)=\sum_{n=0}^{\infty} b_{n}z^{n}$. For some Schwarz function $\varphi$, Theorem \ref{Prop:Operator} and the preceding lemma give
\begin{align*}
\mathcal{M}_{r}(f) &= \mathcal{M}_{r}( h(\varphi) ) = \mathcal{M}_{r} \left( \sum_{n=0}^{\infty} b_{n} \big(\varphi(z)\big)^{n}\right)\\
&\leq \sum_{n= 0}^{\infty} |b_{n}| \left(\mathcal{M}_{r}\big(\varphi (z)\big)\right)^{n} \leq \sum_{n=0}^{\infty} |b_{n}||z|^{n}
\end{align*}
for $0 \leq r \leq 1/3.$ Thus
\[ \mathcal{M}_{r}(f) \leq \mathcal{M}_{r}(h). \qedhere\]

\end{proof}

Theorem \ref{thm:sub} and Theorem \ref{Prop:Operator} give a more general result.

\begin{theorem}\label{thm:Gensub}
Suppose the analytic functions $f, g$ and $h$ satisfy $f(z) = g(z) h(\varphi(z))$, $z\in\mathbb{D}$, for some Schwarz function $\varphi$. Further suppose that $|g(z)| \leq b$ for $|z| < \rho \leq 1$. Then
\begin{equation*}
\mathcal{M}_{r}(f) \leq b \mathcal{M}_{r}(h),
\end{equation*}
for $0 \leq r \leq \rho/3.$
\end{theorem}

\begin{proof}
It follows from Theorem \ref{Prop:Operator} that
\[\mathcal{M}_r(f) \leq \mathcal{M}_r(g) \, \mathcal{M}_r\left(h\big(\varphi\big) \right).\]
Since $g$ is uniformly bounded by $b$ on $\mathbb{D}_{\rho} := \{z \in\mathbb{C} : |z| < \rho\}$, it follows that $|g(z)/b| \leq 1$ for all $z\in\mathbb{D}_{\rho}$. So for $0 < r \leq \rho/3$,  Theorem \ref{Eq1:Bohr_inequality} shows that $\mathcal{M}_r(g) \leq b$. On the other hand, Theorem \ref{thm:sub} yields $\mathcal{M}_{r}(h(\varphi)) \leq \mathcal{M}_{r}(h)$ for $0 < r \leq 1/3$. Hence
\[
\mathcal{M}_r(f) \leq b \mathcal{M}_r(h).
\]
for $0 < r \leq \rho/3.$
\end{proof}

\begin{rem}
Obviously, Theorem \ref{thm:Gensub} reduces to Theorem \ref{thm:sub} in the event $g$ is a constant one function.
\end{rem}

Theorem \ref{thm:Gensub} also readily implies the following quasi-subordination result, first established by Alkhaleefah \emph{et al.} in \cite[Theorem 2.1]{AKP}.

\begin{corollary}\label{thm:Quasisub}
 Suppose $f(z) = g(z)h(\varphi(z))$ in $\mathbb{D}$, where $g$ and $h$ are analytic, $|g(z)| \leq 1$ in $\mathbb{D}$, and $\varphi$ is a Schwarz function. Then
\[\mathcal{M}_{r}(f) \leq \mathcal{M}_{r}(h)\]
for $0 \leq r \leq 1/3$.
\end{corollary}

The following theorem displays a von Neumann-type inequality \eqref{Eq2:vonNeumann} for Schwarz functions.

\begin{theorem}\label{thm:genVN}
 Suppose $h$ is analytic in $\mathbb{D}$ and continuous in $\overline{\mathbb{D}}$. If $0 \leq r \leq 1/3$, then
 $$ \mathcal{M}_{r}\big(h(\varphi)\big)  \leq  \| h \|_{\infty}$$
for every Schwarz function  $\varphi$.
\end{theorem}

\begin{proof}
Let $h(z)=\sum_{n=0}^{\infty} b_{n}z^{n}$. Then Lemma \ref{lem:BohrSchwarz} yields
\[\mathcal{M}_{r}\big(h(\varphi)\big) \leq \sum_{n=0}^{\infty} \left\vert b_{n} \right\vert \left\vert \mathcal{M}_{r}\big(\varphi (z)\big)^{n}\right\vert \leq \sum_{n=0}^{\infty} \left\vert b_{n}\right\vert r^{n} \leq \| h \|_{\infty}\]
for $0 \leq r \leq 1/3$.
\end{proof}

Turning to sections, here is a key result on the $k$-th section $s_k(f) = a_0 + a_1 z + \cdots + a_kz^k.$

\begin{lemma}\label{lem:RogosinskiSchwarz}
Let $\varphi$ be a Schwarz function and $k \in \mathbb{N}_0$. Then
\[\sup_{|z| \leq r} \left| s_k\big(\varphi^j \big) \right| \leq |z|^j, \quad 0 \leq r \leq 1/2.\]
Further,
\[\mathcal{M}_r\left( s_k\big(\varphi^j \big) \right) \leq |z|^j, \quad 0 \leq r \leq 1/3.\]
%\[\| \varphi^j\|_k = \mathcal{M}_r\left( s_k\big(\varphi^j \big) \right) \leq |z|^j, \quad 0 \leq r \leq 1/3.\]
\end{lemma}

\begin{proof}
Since $\left| \left(\varphi(z)/z\right)^j\right| = \left| \varphi(z)/z\right|^j  \leq 1$, Theorem \ref{thm:Rogosinskiradius} shows that
\[\sup_{|z| \leq r} \left| s_k \left(\frac{\varphi}{z}\right)^j   \right| \leq 1\]
for $0 \leq r\leq 1/2$. Further, Theorem \ref{Eq1:Bohr_inequality} gives
\[\mathcal{M}_r\left( \left( \frac{\varphi }{z}\right) ^{j}\right) \leq1 \]
for $0 \leq r \leq 1/3$. Hence
\begin{equation*}
\mathcal{M}_{r}\left( s_{k}\left( \frac{\varphi }{z}\right) ^{j}\right)  \leq \mathcal{M}_r\left( \left( \frac{\varphi }{z}\right) ^{j}\right) \leq 1,
\end{equation*}
which implies
\[\mathcal{M}_{r}\left( s_{k}\left(\varphi ^{j}\right) \right) \leq |z|^{j}. \qedhere \]
\end{proof}

\begin{theorem}\label{thm:RogosinskiSub}
 Suppose $f(z) = \sum_{n=0}^{\infty} a_n z^n \prec h(z) = \sum_{n=0}^{\infty} b_n z^n$ in $\mathbb{D}$. Then for each $k$-th section,
\[\left| s_k(f)\right| \leq \mathcal{M}_r \left( s_k(h)\right), \quad 0 \leq r \leq 1/2,\]
that is,
\[\left| a_0 + a_1 z + \cdots + a_kz^k \right| \leq \sum_{n=0}^{k} |b_n| |z|^n, \quad 0 \leq r \leq 1/2.\]
\end{theorem}

\begin{proof}
Since $f \prec h$, it follows that $f(z) = h\big( \varphi(z)\big) = \sum_{n=0}^{\infty} b_n \varphi^n(z)$ for some Schwarz function $\varphi$.  Then
\begin{equation*}
s_k\big( f\big) = \sum_{n=0}^{k} b_n s_k \big( \varphi^n \big).
\end{equation*}
By Lemma \ref{lem:RogosinskiSchwarz},
\[| s_k\big( f\big)| \leq \sum_{n=0}^{k} |b_n| | s_k \big( \varphi^n \big)| \leq \sum_{n=0}^{k} |b_n| |z|^n = \mathcal{M}_r \big( s_k(h) \big)
\]
whenever $|z| \leq 1/2$.
%Let $\omega(z) = \sum_{n=1}^{\infty} \alpha_n z^n$ be the Schwarz function such that $f(z) = g\left( \omega(z)\right)$. For $0 \leq k \leq n$, write
%\[
%\omega^{(k)}(z) = \sum_{n=k}^{\infty} \alpha_n^{(k)} z^n,
%\]
%where $\omega^0(z) = 1 = \sum_{n=0}^{\infty} \alpha_n^{(k)} z^n$, that is, $\alpha_0^{(0)} = 1$, $\alpha_n^{(0)} = 0$ for $n \geq 1$. Then
%\begin{align*}
%g \big( \omega(z)\big) &= \sum_{k=0}^{\infty} b_k \left( \sum_{n=k}^{\infty} \alpha_n^{(k)} z^n\right)\\
%&= \sum_{k=0}^{\infty} \left( \sum_{s=0}^{k} b_s \alpha_k^{(s)} z^k\right)\\
%&= \sum_{k=0}^{\infty} B_k z^k,
%\end{align*}
%where $B_k = \sum_{s=0}^{k} b_s \alpha_k^{(s)}$. By the Schwarz lemma, we have $\left|\frac{\omega^k(z)}{z^k} \right| < 1$, which implies that the $m$-th partial sum
%\[
%\left| \sum_{n=k}^{m} \alpha_n^{(k)} z^{n-k}\right| < 1, \qquad |z| \leq 1/2.
%\]
%Hence, for $|z| = r \leq 1/2$,
%\begin{align*}
%\big| s_m(f) \big| &= \big|s_m\big( g \circ \omega\big) \big| = \left| \sum_{k=0}^{m} B_k z^k\right|\\
%&= \left| \sum_{k=0}^{m} \left( \sum_{n=0}^{k} b_n \alpha_k^{(n)}\right) z^k \right|\\
%&= \left| \sum_{k=0}^{\infty} b_k \left( \sum_{s=k}^{m} \alpha_s^{(k)} z^{s-k}\right) z^k \right|\\
%&\leq \sum_{k=0}^{m} |b_k| |z|^k\\
%&= \mathcal{M}_r \big( s_m(g)\big). \qedhere
%\end{align*}
\end{proof}

%If we define an operator $\mathcal{M}$ on $\mathcal{A}$ by
%\[
%\mathcal{M}(f) := \sum_{n=0}^{\infty} |a_n| z^n,
%\]
%where $f(z) = \sum_{n=0}^{\infty} a_n z^n \in\mathcal{A}$, then we have the next result.

\bigskip
%The next theorem is related to Rogosinski's inequality for subordination.
Note that it follows easily from Theorem \ref{thm:sub} that whenever $f \prec h$, then
\[\mathcal{M}_r \left( s_k (f) \right) \leq \mathcal{M}_r (f) \leq \mathcal{M}_r (h)\]
for $0 \leq r < 1/3$. Indeed, we show next that $\mathcal{M}_r \left( s_k (f) \right) \leq \mathcal{M}_r \left( s_k (h) \right).$

%Recall that by Bohr's theorem (Theorem \ref{Eq1:Bohr_inequality}???), if $\left\Vert w(z)\right\Vert_{\infty} = \left\Vert \sum_{n=0}^{\infty} c_n z^n\right\Vert_{\infty} \leq 1$, then $\left\Vert \mathcal{M}(w(z))\right\Vert_{\infty }$ $\leq 1$ on the disk $|z|<1/3$.

\begin{theorem}\label{thm:RogosinskiSub2}
If $f \prec h$ in $\mathbb{D},$ then each section satisfies
\[\mathcal{M}_{r}\left( s_{k}(f) \right) \leq \mathcal{M}_{r}(s_{k}(h))\]
for $0 \leq r \leq 1/3$.
\end{theorem}

\begin{proof}
Let  $f(z) = h\big( \varphi(z)\big) = \sum_{n=0}^{\infty} b_n \varphi^n(z)$ for some Schwarz function $\varphi$. Then
\begin{equation*}
s_k\big( f\big) = \sum_{n=0}^{k} b_n s_k \big( \varphi^n \big). %\qquad z\in\mathbb{U}.
\end{equation*}
By the subadditivity of $\mathcal{M}_r$ and Lemma \ref{lem:RogosinskiSchwarz},
\[
\mathcal{M}_r \left( s_k(f)\right) \leq \sum_{n=0}^{k} |b_n| \mathcal{M}_r \left( s_k \left( \varphi^n\right)\right) \leq \sum_{n=0}^{k} |b_n||z|^n = \mathcal{M}_r \big( s_k(h)\big). \qedhere
\]
\end{proof}

%\begin{proof}
%\begin{enumerate}
%\item [(a)] Since $f(z) = \sum_{n=0}^{\infty} b_n \varphi^n(z)$, it follows that
%\begin{equation*}
%s_k\big( f\big) = \sum_{n=0}^{k} b_n \varphi^n(z) = \sum_{n=0}^{\infty} b_n s_k \big( \varphi^n \big), \qquad z\in\mathbb{U}.
%\end{equation*}
%Then
%\[
%\mathcal{M}_r \big(s_k(f)\big) = \sum_{n=0}^{\infty} |b_n| \left|s_k \big( \varphi^n \big)\right| \leq \sum_{n=0}^{\infty} |b_n| \mathcal{M}_r \left( s_k \big( \varphi^{n}\big) \right).
%\]
%By the subadditivity of $\mathcal{M}_r$ and Lemma \ref{lem:BohrRogosinskiSchwarz}, for $0 < |z| = r \leq 1/3$, we get
%\begin{equation*}
%\mathcal{M}_{r}\left( s_{k}\big( f\big) \right) \leq
%\sum_{n=0}^{\infty} |b_{n}|\, \mathcal{M}_{r} \left(s_{k}\big(\varphi ^{n}\big)\right) \leq \sum_{n=0}^{\infty} |b_{n}| |z|^n = \mathcal{M}_r (h).  %\mathcal{M}_r \left( s_{k}\big(h(z)\big)\right). \qedhere
%\end{equation*}
%
%\item [(b)] Note that  $f_k(z) = \sum_{n=0}^{k} b_n \varphi^n(z)$. Then
%\[
%\mathcal{M}_r (f_k) = \sum_{n=0}^{k} |b_n| \mathcal{M}_r \left( s_k \big( \varphi^{n}\big) \right).
%\]
%By Lemma \ref{lem:BohrRogosinskiSchwarz}, for $0 < |z| = r \leq 1/3$, we get
%\begin{equation*}
%\mathcal{M}_{r}\left( f_k \right) =
%\sum_{n=0}^{k} |b_{n}|\, \mathcal{M}_{r} \left(s_{k}\big(\varphi ^{n}\big)\right) \leq \sum_{n=0}^{k} |b_{n}| |z|^n = \mathcal{M}_r (s_{k}(h)). \qedhere
%\end{equation*}
%\end{enumerate}
%\end{proof}

\begin{corollary}
If $f(z) = \sum_{n=0}^{\infty} a_{n}z^{n} \prec h(z) = \sum_{n=0}^{\infty} b_{n}z^{n},$ and $h$ is univalent in $\mathbb{D}$, then
\[\left \vert \sum_{n=0}^{k} a_{n} z^{n} \right \vert \leq \left| b_0\right| + \frac{k(k+1)}{2}\left|b_1\right|\]
for $|z| \leq 1/2$, and
\[\sum_{n=0}^{k} \left\vert a_{n}\right\vert |z|^{n} \leq \left| b_0\right| + \frac{k(k+1)}{2}\left|b_1\right|\]
for $|z|\leq 1/3$.
\end{corollary}

\begin{proof}
By de Branges Theorem \cite{deBranges}, it follows that $\left\vert b_{n}\right\vert \leq n\left|b_1\right|$ for $n\in\mathbb{N}$. Hence
\[\sum_{n=0}^{k} \left| b_n\right| |z|^n \leq \left| b_0\right| + \left|b_1\right|\sum_{n=1}^{k} n = \left| b_0\right| + \frac{k(k+1)}{2}\left|b_1\right|.\]
The result now readily follows from Theorem \ref{thm:RogosinskiSub} and Theorem \ref{thm:RogosinskiSub2}.
\end{proof}

\bigskip
\subsection*{\textbf{Acknowledgment.}} This work was initiated during the visit of the first author to USM. The second and third authors gratefully acknowledge support from USM research university grants 1001.PMATHS.8011101 and 1001.PMATHS.8011038.

\end{document}